\documentclass[12pt]{article}
\usepackage{graphicx} 
\usepackage{amsmath,amsthm,amssymb,amsfonts}
\usepackage{indentfirst}
\usepackage{geometry}
\usepackage{graphicx}
\usepackage{cite}
\usepackage{abstract}

\newtheorem{theorem}{Theorem}[section]
\newtheorem{proposition}{Proposition}[section]

\newtheorem{lemma}{Lemma}[section]

\allowdisplaybreaks[4]

\title{On the First Eigenvalue of Embedded Minimal Hypersurfaces in the Unit Sphere}
\author{Jinhong Yu }
\date{}

\begin{document}

\maketitle

\begin{abstract}
 Let $\Sigma$ be a closed embedded minimal hypersurface in $\mathbb{S}^{n+1}$. We establish an improved lower bound for the first non-zero eigenvalue of the induced Laplace-Beltrami operator on $\Sigma$. It is slightly better than the bound of Duncan-Sire-Spruck \cite{[1]}.
\end{abstract}

\section{Introduction}

A famous conjecture of Yau states that for any closed embedded minimal hypersurface $\Sigma$ in the unit sphere $\mathbb{S}^{n+1}$ the first eigenvalue $\lambda_1(\Sigma)=n$ \cite{[4]}. Choosing the linear function as test function shows that $\lambda_1(\Sigma)\le n$, so Yau's conjecture is equivalent to that the lower bound $\lambda_1(\Sigma)$ is exactly $n$. The conjecture has been verified for certain special cases: Lawson surfaces and Karcher-Pinkall-Sterling examples by Choe-Soret \cite{[7]} and isoparametric hypersurfaces by Tang-Yan \cite{[6]}. For a general minimal hypersurface, a seminal progress to the lower bound of $\lambda_1(\Sigma)$ was made by Choi-Wang \cite{[2]}, who proved $\lambda_1(\Sigma)\ge \frac{n}{2}$. In \cite{[8]} Brendle showed that equality in the Choi-Wang estimate can never happen. Recently, Duncan-Sire-Spruck proved the following important quantitative improvement:

\begin{theorem}[\cite{[1]}]
Let $\Sigma^n \subset \mathbb{S}^{n+1}$ be a closed embedded minimal hypersurface and denote $\Lambda=\max_{\Sigma}\|A\|$ where $A$ is the second fundamental form of $\Sigma$. Then there are constants
\begin{equation}
a_n \ge \frac{(n-1)n^2}{32000}\;\;and\;\; b_n \le \frac{5n^2}{216}
\end{equation}
such that
\begin{equation}\label{DSS}
    \lambda_1(\Sigma)\ge \frac{n}{2}+\frac{a_n}{\Lambda^6+b_n}.
\end{equation}
\end{theorem}

In \cite{[10]}, Jimnez-Chinchay-Zhou give another improvement
\begin{equation}\label{JCZ}
\lambda_1(\Sigma)\ge \frac{n}{2}+\frac{n\left(n+1 \right)}{32\left(12\Lambda+n+11 \right)^2+8}.
\end{equation}
Other estimates are obtained by Zhao in \cite{[9]} when $n = 2$. In this paper, we present a new improvement of Choi-Wang result.

\begin{theorem}\label{main theorem}
Let $\Sigma^n \subset \mathbb{S}^{n+1}$ be any closed embedded minimal hypersurface. Then
\begin{equation}\label{Yu estimate}
\lambda_1(\Sigma) \ge \frac{n}{2}+\frac{n^2}{27 \left(\sqrt{\Lambda^2+2n}+\Lambda\right)^2}.
\end{equation}
\end{theorem}

We remark that by Simons gap theorem the constant $\Lambda\ge\sqrt{n}$ unless $\Sigma$ is totally geodesic \cite{[3]}. So, in general, the error term in (\ref{JCZ}) is strictly less than $\frac{n\left(n+1\right)}{5408n^2}$ which is of 0-th order in $n$. Our estimate in (\ref{Yu estimate}) seems to be of the $1$-th order in $n$, so it is much better in the cases $\Lambda\approx\sqrt{n}$ and $n$ large. For example, (\ref{Yu estimate}) implies a trivial estimate
\begin{equation}\label{Yu estimate 2}
\lambda_1(\Sigma) \ge \frac{n}{2}+\frac{n^2}{270 \Lambda^2}.
\end{equation}
The constant $\frac{1}{270}$ is also better than the constants in (\ref{DSS}) and (\ref{JCZ}).

 \section{Proof of Theorem \ref{main theorem}}

We start with the Reilly formula for submanifolds.

Let $(M^{n+1},g)$ be a smooth orientable Riemannian manifold with boundary $\Sigma:=\partial M^{n+1}$. Let $\nu$ denote the inward normal. Let $dv_g$ denote the volume element on $(M^{n+1},g)$ and $dS_g$ denote the volume element of the induced metric on $\Sigma$. Let $A$ denote the second fundamental form of $\Sigma$ with respect to $\nu$ and $H$ denote the associated mean curvature. For a smooth function $u$, let $u_{\nu}$ denote the inward normal derivative of $u$ on $\Sigma$ and $\triangledown^{\Sigma} u$ denote the gradient of $u$ with respect to the induced metric on $\Sigma$.

\begin{lemma}[\cite{[11]}]
Assume as above. For $u\in C^2(\overline{M} )$,
   \begin{equation}
       \int_{M} \left(\Delta u\right)^2-|\nabla^2 u|^2  dv_g=\int_{M} Ric_M(\nabla u,\nabla u) dv_g-
       \int_{\Sigma} \left(2\Delta^{\Sigma}u +Hu_{\nu}\right)u_{\nu}dS_g -\int_{\Sigma} A(\nabla^{\Sigma}u,\nabla^{\Sigma}u)dS_g.
   \end{equation}
\end{lemma}
\hspace{\fill}

From now on, let $\Sigma$ be a minimal smooth hypersurface in $S^{n+1}$ that divides $S^{n+1}$ into two components $M_1$ and $M_2$. Denote by $\Psi$
a normalized eigenfunction corresponding to the first nonzero eigenvalue $\lambda_1(\Sigma)$, so that
$$-\Delta^{\Sigma} \Psi=\lambda_1 \Psi,\qquad \|\Psi \|_{L^2(\Sigma)}=1. $$
Let $u$ be the unique solution to
\begin{equation}\label{equ}
    \begin{cases}
        \Delta u=0,\;\;\;\;\mbox{ in } \;\;M_1,\\
        u=\Psi, \;\;\;\;\;\;\mbox{ on }\;\;\Sigma.
    \end{cases}
\end{equation}
We fix the orientation on $\Sigma$ pointing into $M_1$. We may assume that
$$-\int_{\Sigma} A(\nabla^{\Sigma}u,\nabla^{\Sigma}u)dS_g\ge 0,$$
otherwise we work on $M_2$ instead. Then by Reilly's formula and $H=0$, the solution $u$ to (\ref{equ}) satisfies
\begin{equation*}
    \begin{aligned}
        -\int_{M_1} | \nabla^2 u|^2dv_g &\ge n\int_{M_1} |\nabla u|^2dv_g-2\int_{\Sigma}u_{\nu} \Delta^{\Sigma}u dS_g\\
        &=n\int_{M_1} |\nabla u|^2dv_g+2\lambda_1\int_{\Sigma} u_{\nu}u dS_g\\
        &=n\int_{M_1} |\nabla u|^2dv_g-2\lambda_1\int_{M_1} |\nabla u|^2 dv_g.
    \end{aligned}
\end{equation*}
So,
 \begin{equation}\label{eq7}
 (2\lambda_1-n)\int_{M_1} |\nabla u|^2 dv_g\ge \int_{M_1} |\nabla^2 u|^2 dv_g.
 \end{equation}

The core in the proof is to establish a lower bound for the Hessian energy $\int_{M_1}|\nabla^2 u|^2 dv_g$.

\subsection{Preliminary energy estimate}

Let $\|A\|$ be the norm of the second fundamental form. Let $\Lambda=\sup_\Sigma\|A\|$. By Simons inequality \cite{[3]}
$$\int_{\Sigma} \| A\|^2(\|A\|^2 - n) dS_g \ge 0.$$
If $\Lambda^2<n$, then $\Sigma$ is totally geodesic, so $\lambda_1(\Sigma)=n$. For the remainder of the proof, we may assume that $\Lambda^2\ge n$.

We denote $M^{t}=\{x\in M_1 | d(x)>t\}$, where $d(x)$ is the distance of $x$ to $\Sigma$ in $M_1$. Then $\partial M^{t} = \Sigma^{t}$ is a smooth embedded hypersurface for $0\le t<\arctan \frac{1}{\Lambda}$. Let $H_{\Sigma^t}$ be the mean curvature of $\Sigma^t$ with respect to the inward normal of $M^t$. According to Lemma 3.5 in \cite{[1]}, the mean curvature has upper and lower  bounds
\begin{equation}\label{H}
    0\le H_{\Sigma^t}\le 2\Lambda,\qquad 0\le t\le\arctan{\frac{1}{2\Lambda}}.
\end{equation}

\begin{lemma}
   Let $\eta(t)=\int_{M^t} |\nabla u|^2\;dv_g$ and $C_1=2\lambda_1\eta(0)$, then $\eta(t)$ satisfies the  inequality:
   \begin{equation}\label{eq9}
       \eta^{\prime\prime}+2\Lambda \eta^{\prime}\le 2C_1.
   \end{equation}
\end{lemma}
\begin{proof}
We compute the derivatives of $\eta(t)$:
\begin{align*}
    \eta^{\prime}(t)&= -\int_{\Sigma^t} |\nabla u|^2 dS_g,\\
    \eta^{\prime\prime}(t)&=-\left[\int_{\Sigma^t} \langle \nabla|\nabla u|^2 , \nabla d \rangle dS_g -\int_{\Sigma^t} |\nabla u|^2 H _{\Sigma^t} dS_g\right].
\end{align*}
Applying the equation (\ref{eq7}) and (\ref{H}), we obtain that
  \begin{align*}
      \eta^{\prime\prime}(t)&=\int_{M^t} \triangle |\nabla u|^2 dv_g + \int_{\Sigma^t} |\nabla u|^2 H_{\Sigma^t} dS_g\\
    &= \int_{M^t} \left(2|\nabla^2 u|^2+2n|\nabla u|^2 \right)dv_g +  \int_{\Sigma^t} |\nabla u|^2 H_{\Sigma^t}dS_g\\
    &\le 2\int_{M_1}  \left(|\nabla^2 u|^2 + n|\nabla u|^2 \right) dv_g - 2\Lambda \eta^{\prime}(t)\\
  &\le 4\lambda_1 \int_{M_1} |\nabla u|^2 dv_g - 2\Lambda \eta^{\prime}(t).
  \end{align*}
The required formula follows.
\end{proof}

\begin{lemma}\label{p2.1}
Assume as before, then, for any $0\le t \le \arctan\frac{1}{2\Lambda}$, we have 
\begin{equation}\label{eta derivative}
\eta^{\prime}(t)\ge -2\left(\Lambda+\sqrt{\Lambda^2+2\lambda_1}\right)\eta(0).
\end{equation}
\end{lemma}
\begin{proof}
Multiplying both sides of (\ref{eq9}) by $e^{2\Lambda t}$ and integrate t from 0 to s, we obtain that
\begin{equation*}\label{eq10}
    \int_{0}^{s} \eta(t)^{\prime\prime} e^{2\Lambda t} dt + \int_{0}^{s} 2\Lambda \eta(t)^{\prime} e^{2\Lambda t} dt \le 2C_1 \int_{0}^{s} e^{2\Lambda t} dt.
\end{equation*}
Direct computation shows
\begin{equation*}
    e^{2\Lambda s} \eta(s)^{\prime} - \eta(0)^{\prime} \le 2\lambda_1\eta(0) \frac{e^{2\Lambda s}-1}{\Lambda}.
    \end{equation*}
So,
$$\eta(s)^{\prime} \le 2\lambda_1\eta(0) \frac{1-e^{-2\Lambda s}}{\Lambda}+e^{-2\Lambda s}\eta(0)^{\prime} .$$
Integrating once again we get, for any $0\le t\le\arctan\frac{1}{2\Lambda}$,
$$\eta(t)-\eta(0)\le \frac{2\lambda_1\eta(0)}{\Lambda}\left(t-\frac{1-e^{-2\Lambda t}}{2\Lambda}\right)+\frac{1-e^{-2\Lambda t}}{2\Lambda}\eta(0)^{\prime}.$$
Dropping the term $\eta(t)$ we have
$$\eta'(0)\ge -2\eta(0)\cdot\left(\frac{\Lambda+2\lambda_1t}{1-e^{-2\Lambda t}}-\frac{\lambda_1}{\Lambda}\right).$$
Notice that whenever $t\le\arctan\frac{1}{2\Lambda}$,
$$\frac{\Lambda+2\lambda_1t}{1-e^{-2\Lambda t}}-\frac{\lambda_1}{\Lambda}\le\frac{\Lambda+2\lambda_1t}{2\Lambda t(1-\Lambda t)}-\frac{\lambda_1}{\Lambda}
=\frac{1+2\lambda_1 t^2}{2t(1-\Lambda t)}.$$
The right hand side achieves the minimum at
$$t_0=\frac{1}{\Lambda+\sqrt{\Lambda^2+2\lambda_1}}\le\arctan\frac{1}{2\Lambda}.$$
Substituting into the formulas we have
$$\eta'(0)\ge -2\eta(0)\cdot\frac{1+2\lambda_1 t_0^2}{2t_0(1-\Lambda t_0)}=-2\left(\Lambda+\sqrt{\Lambda^2+2\lambda_1}\right)\eta(0).$$
Since$$\eta^{\prime\prime}(t)=2\int_{M^t}\left( |\nabla^2 u|^2+n|\nabla u^2| \right)\;dv_g + \int_{\Sigma^t}|\nabla u|^2 H_{\Sigma^t}\;dS_g \ge 0,$$
we have, for any $0\le t \le \arctan\frac{1}{2\Lambda}$,
$$\eta^{\prime}(t)\ge \eta^{\prime}(0)\ge -2\left( \Lambda+\sqrt{\Lambda^2+2\lambda_1}\right)\eta(0).$$
This completes the proof of the lemma.
\end{proof}

\subsection{Hessian energy estimate and the proof of Theorem \ref{main theorem}}

The key step is the following Hessian estimate.

\begin{proposition}\label{p2.2}
    Assume as before. Then
    \begin{equation}\label{eq11}
        \int_{M_1} |\nabla^2u|^2 dv_g\ge \frac{2n^2}{27\left( \sqrt{\Lambda^2+2\lambda_1}+\Lambda\right)^2 }\int_{M_1}|\nabla u|^2 dv_g.
    \end{equation}
\end{proposition}
\begin{proof}
Integrating the formula (\ref{eta derivative}) we have
\begin{equation}\label{eq12}
    \eta(s)\ge \left[1-2( \sqrt{\Lambda^2+2\lambda_1}+\Lambda)s\right]\eta(0).
\end{equation}
A simple calculation
\begin{align*}
2\int_0^s  \int_{M^t} |\nabla^2 u|^2 dv_g\;dt+2n\int_0^s \int_{M^t}|\nabla u|^2\;dv_g \;dt
    &=\int_0^s \int_{M^t}\triangle |\nabla u|^2 \;dv_g\;dt\\
    &=-\int_0^s\int_{\Sigma^t} \langle \nabla |\nabla u|^2 \;,\;\nabla d  \rangle\;dS_g\;dt\\
    &=-\int_{M_1 \setminus M^s}  \; \langle\; \nabla|\nabla u|^2\; , \nabla d\;  \rangle\;dv_g
\end{align*}
yields
\begin{equation*}\label{eq13}
     \int_0^s  \int_{M^t} |\nabla^2 u|^2 dv_g\;dt+n\int_0^s \eta(t) \;dt=-\int_{M_1 \setminus M^s}  \; \langle\; \langle \nabla\nabla u\;,\;\nabla u \rangle \; , \nabla d\;  \rangle\;dv_g.
\end{equation*}
Dropping the first term on the left hand side and applying the Cauchy-Schwarz inequality, we obtain that
\begin{equation*}
   n\int_0^s\eta(t) \;dt \le \left( \int_{M_1\setminus M^s}  |\nabla^2 u|^2 \;dv_g \right)^{\frac12}\left(\int_{M_1\setminus M^s} |\nabla u|^2\; dv_g \right)^{\frac12}.
\end{equation*}
Combining with the inequality (\ref{eq12}), we have the following bound
\begin{align*}
    n\eta(0)\left[s-\left(\sqrt{\Lambda^2+2\lambda_1}+\Lambda\right)s^2\right]&
    \le n\int_0^s\eta(t) \;dt\\
    &\le \left( \int_{M_1\setminus M^s}  |\nabla^2 u|^2 \;dv_g \right)^{\frac12}\left(\int_{M_1\setminus M^s} |\nabla u|^2\; dv_g \right)^{\frac12}\\
    &\le \left( \int_{M_1} |\nabla^2 u|^2 \; dv_g \right)^{\frac12}
    \left[2\left(\sqrt{\Lambda^2+2\lambda_1}+\Lambda\right)\eta(0) s\right]^{\frac12}.
\end{align*}
Obviously $\frac{1}{\sqrt{\Lambda^2+2\lambda_1}+\Lambda}\le\arctan\frac{1}{2\Lambda}$, so it follows 
\begin{equation}\label{eq15}
    \int_{M_1}|\nabla^2 u|^2 \; dv_g\ge n^2 s\eta(0)\frac{\left(1-(\sqrt{\Lambda^2+2\lambda_1}+\Lambda)s\right)^2}{2(\sqrt{\Lambda^2+2\lambda_1}+\Lambda)},
    \;\;\;\;\forall s \in \left[ 0\;, \; \frac{1}{\sqrt{\Lambda^2+2\lambda_1}+\Lambda} \right].
\end{equation}
Let
\begin{equation*}
    h(s)=s\left[1-(\sqrt{\Lambda^2+2\lambda_1}+\Lambda)s\right]^2, \;\;\; s \in \left[ 0\;,\;\frac{1}{\sqrt{\Lambda^2+2\lambda_1}+\Lambda} \right].
\end{equation*}
We have
\begin{equation*}
    \max_{s\in \left[0\;, \; \frac{1}{\sqrt{\Lambda^2+2\lambda_1}+\Lambda} \right]}h(s)=\frac{1}{3(\sqrt{\Lambda^2+2\lambda_1}+\Lambda)}\left( \frac23\right)^2.
\end{equation*}
Substituting into (\ref{eq15}) gives (\ref{eq11}).
\end{proof}

\begin{proof}[Proof of theorem1.2]
  According to proposition \ref{p2.2} and equation (\ref{eq7}), we obtain that
    \begin{align*}
        \left(2\lambda_1 - n\right)\int_{M_1} |\nabla u|^2 dv_g&\ge \int_{M_1} |\nabla^2 u|^2 dv_g \\
        &\ge \frac{2n^2}{27\left( \sqrt{\Lambda^2+2\lambda_1}+\Lambda\right)^2}\int_{M_1}|\nabla u|^2 dv_g\;.
    \end{align*}
   This shows that
   \begin{equation*}
   \lambda_1 \ge \frac{n}{2}+  \frac{n^2}{27\left( \sqrt{\Lambda^2+2\lambda_1}+\Lambda\right)^2}.
   \end{equation*}
   Noticing that $\lambda_1\le n$, the proof of Theorem 1.2 is completed.
   \end{proof}

\nocite{*}
\bibliographystyle{plain}
\bibliography{[1]}

\begin{thebibliography}{10}

\bibitem{[8]}
Simon Brendle.
\newblock Minimal surfaces in: a survey of recent results.
\newblock {\em Bulletin of Mathematical Sciences}, 3(1):133--171, 2013.

\bibitem{[7]}
Jaigyoung Choe and Marc Soret.
\newblock First eigenvalue of symmetric minimal surfaces in $\mathbb{S}^3$.
\newblock {\em Indiana University mathematics journal}, pages 269--281, 2009.

\bibitem{[2]}
Hyeong~In Choi and Ai~Nung Wang.
\newblock A first eigenvalue estimate for minimal hypersurfaces.
\newblock {\em Journal of differential geometry}, 18(3):559--562, 1983.

\bibitem{[1]}
Jonah~AJ Duncan, Yannick Sire, and Joel Spruck.
\newblock An improved eigenvalue estimate for embedded minimal hypersurfaces in the sphere.
\newblock {\em International Mathematics Research Notices}, 2024(18):12556--12567, 2024.

\bibitem{[10]}
Asun Jim{\'e}nez, Carlos~Tapia Chinchay, and Detang Zhou.
\newblock A lower bound for the first eigenvalue of a minimal hypersurface in the sphere, 2026.

\bibitem{[11]}
Robert~C Reilly.
\newblock {Applications of the Hessian operator in a Riemannian manifold}.
\newblock {\em Indiana University Mathematics Journal}, 26(3):459--472, 1977.

\bibitem{[3]}
James Simons.
\newblock Minimal varieties in riemannian manifolds.
\newblock {\em Annals of Mathematics}, 88(1):62--105, 1968.

\bibitem{[6]}
Zizhou Tang and Wenjiao Yan.
\newblock Isoparametric foliation and yau conjecture on the first eigenvalue.
\newblock {\em Journal of Differential Geometry}, 94(3):521--540, 2013.

\bibitem{[5]}
Paul~C Yang and Shing-Tung Yau.
\newblock {Eigenvalues of the Laplacian of compact Riemann surfaces and minimal submanifolds}.
\newblock {\em Annali della Scuola Normale Superiore di Pisa-Classe di Scienze}, 7(1):55--63, 1980.

\bibitem{[4]}
Shing-Tung Yau.
\newblock {\em Seminar on differential geometry}.
\newblock Number 102. Princeton University Press, 1982.

\bibitem{[9]}
Yuhang Zhao.
\newblock First eigenvalue of embedded minimal surfaces in $\mathbb{S}^3$, 2023.

\end{thebibliography}

\vspace{1cm}
\noindent
SCHOOL OF MATHEMATICAL SCIENCES, CAPITAL NORMAL UNIVERSITY \\
NO. 105 XISANHUAN ROAD, HAIDIAN DISTRICT, BEIJING \\
100048, CHINA \\
Email address: \texttt{2200501011@cnu.edu.cn}
\end{document}